\documentclass[reqno,12pt]{amsart}
\makeatletter
\usepackage{fullpage}
\usepackage{textcmds}  
\usepackage{amsmath, amssymb, amsfonts, amstext, verbatim, amsthm, mathrsfs, stmaryrd}
\usepackage[mathcal]{eucal}
\usepackage[greek,english]{babel}
\usepackage[all,cmtip,2cell]{xy}
\usepackage{pgf,tikz,pgfplots}
\pgfplotsset{compat=1.15}
\usetikzlibrary{arrows}
\usetikzlibrary{positioning}
\usetikzlibrary{quotes}
\usepackage{tikz-cd}

\usepackage{enumerate}
\usepackage{enumitem}
\usepackage[colorlinks=true,linkcolor=blue,citecolor=blue,urlcolor=blue,citebordercolor={0 0 1},urlbordercolor={0 0 1},linkbordercolor={0 0 1}]{hyperref} 
\usepackage[shortalphabetic]{amsrefs} 
\usepackage[nameinlink]{cleveref}

\usepackage{enotez}
\setenotez{backref=true}

\def\makeCal#1{%
\expandafter\newcommand\csname c#1\endcsname{\mathcal{#1}}}
\def\makeBB#1{%
\expandafter\newcommand\csname b#1\endcsname{\mathbb{#1}}}
\def\makeFrak#1{%
\expandafter\newcommand\csname f#1\endcsname{\mathfrak{#1}}}

\count@=0
\loop
\advance\count@ 1
\edef\y{\@Alph\count@}%
\expandafter\makeCal\y
\expandafter\makeBB\y
\expandafter\makeFrak\y
\ifnum\count@<26
\repeat

\theoremstyle{plain}
\newtheorem{thm}{Theorem}[section]
\newtheorem{cor}[thm]{Corollary}
\newtheorem{lem}[thm]{Lemma}

\newtheorem{prop}[thm]{Proposition}
\newtheorem{quest}[thm]{Question}

\theoremstyle{definition}
\newtheorem{rem}[thm]{Remark}
\newtheorem{defn}[thm]{Definition}

\newtheorem{ex}[thm]{Example}

\newtheorem*{thm*}{Theorem}
\newtheorem*{prop*}{Proposition}

\newtheorem{innercustomthm}{Theorem}
\newenvironment{customthm}[1]
  {\renewcommand\theinnercustomthm{#1}\innercustomthm}
  {\endinnercustomthm}

\def\rm{\mathrm}

\DeclareMathOperator{\rk}{rank}

\DeclareMathOperator{\DCoh}{D^b_{coh}}
\DeclareMathOperator{\Dqc}{D_{qc}}
\DeclareMathOperator{\Coh}{Coh}

\DeclareMathOperator{\Cone}{Cone}

\DeclareMathOperator{\Ext}{Ext}

\DeclareMathOperator{\Hom}{Hom}
\newcommand{\id}{\mathrm{id}}

\DeclareMathOperator{\rep}{rep}

\DeclareMathOperator{\QCoh}{QCoh}

\DeclareMathOperator{\Spec}{Spec}

\DeclareMathOperator{\Pic}{Pic}

\DeclareMathOperator{\fib}{fib}

\def\db{\mathrm{D}^{\mathrm{b}}}
\def\bf{\mathbf}

\usepackage{pbox}
\usepackage[normalem]{ulem}

\makeatletter

\usepackage{babel}

\begin{document}

\title{Admissible subcategories of noncommutative curves}

\author{Antonios-Alexandros Robotis}
\address{Cornell University, Department of Mathematics\newline
\indent 212 Garden Avenue, Ithaca, NY, 14853}
\email{ar2377@cornell.edu}

\begin{abstract}
    We study admissible subcategories of the derived categories of smooth noncommutative (nc) curves as classified by Reiten-van den Bergh. We prove that any admissible subcategory of the derived category of a smooth nc curve is again the derived category of a smooth nc curve. We use this result to classify semi\-orthogonal decompositions in derived categories of nc curves. The results obtained imply that phantom categories do not exist in these cases. As a further application, we prove an extension of the Bondal-Orlov reconstruction theorem to the case of orbifold curves. 
\end{abstract}

\maketitle

\tableofcontents

\section{Introduction}

In algebraic geometry and representation theory, triangulated categories are by now widely recognized as fundamental objects. Originally defined as book-keeping devices, it was later realized that they are interesting objects in and of themselves \cites{Beilinson1978,BGG78,MukaiAbelian}. Subsequently, there was an explosion of interest in the 90's catalyzed by Kontsevich's homological mirror symmetry proposal \cite{kontsevich1994homological}, which -- roughly speaking -- formulates mirror symmetry as an equivalence of triangulated categories between certain derived Fukaya categories of symplectic manifolds and derived categories of coherent sheaves of complex projective varieties. 

Semiorthogonal decompositions, as introduced in \cite{B-KSerre}, are a way to understand triangulated categories by breaking them into smaller pieces:
\begin{defn}
\label{D:SOD}
    A \emph{semiorthogonal decomposition} of a triangulated category $\cD$ is a collection $\cA_1,\ldots, \cA_n$ of full additive subcategories that together generate $\cD$ as a triangulated category and such that $i<j$ implies that $\Hom_{\cD}(X_j,X_i) = 0$ for all $X_i\in \cA_i$ and $X_j\in \cA_j$.
\end{defn}

A semiorthogonal decomposition as in \Cref{D:SOD} is written as $\cD = \langle \cA_1,\ldots, \cA_n\rangle$. The constituent pieces $\cA_i$ are called \emph{admissible subcategories} of $\cD$ and can also be characterized as in \Cref{D:admissible}. In the presence of a semiorthogonal decomposition, one obtains direct sum decompositions of additive invariants so that for example 
\[
    \rm{K}_0(\cD) = \bigoplus_{i=1}^n \rm{K}_0(\cA_i).
\]

An important special case is when the factors are as simple as possible, i.e. when $\cA_i \simeq \DCoh(\Spec \bf{C})$ for all $i=1,\ldots, n$. Choosing an indecomposable object $E_i$ from each $\cA_i$ in this case, we obtain a \emph{full exceptional collection} $E_1,\ldots, E_n$ of $\cD$. That is, a collection of objects $E_1,\ldots, E_n$ that are \emph{exceptional} in that 
\[
    \Hom(E_i,E_i[j]) = 
    \begin{cases}
        \bf{C}& \text{ if }j =0\\
        0 &\:\text{else},
    \end{cases}
\]
such that furthermore 
\begin{enumerate}
    \item $E_1,\ldots, E_n$ generate $\cD$ as a triangulated category; and
    \item and $i<j$ $\Rightarrow$ $\Hom_{\cD}(E_j,E_i) = 0$.
\end{enumerate} 

The study of full exceptional collections in bounded derived categories of varieties and finite dimensional representations of finite dimensional algebras is by now classical, initiated by the pioneering works of Beilinson \cite{Beilinson1978}, Kapranov \cites{KapranovtypeA,KapranovGr}, and Rudakov's seminar \cite{Rudakov1990}.

\begin{ex}
    The archetypal example of a semiorthogonal decomposition comes from Beilinson's \cite{Beilinson1978} full exceptional collection of line bundles on $\bf{P}^n$, given by $\cO,\cO(1),\ldots, \cO(n)$. 
\end{ex}

 There is an action of the braid group $\mathfrak{B}_n$ on the set of length $n$ semiorthogonal decompositions of $\cD$ by mutation constructed in \cite{B-KSerre}, which induces a braid group action on the set of full exceptional collections of length $n$. The orbits of this action are commonly referred to as \emph{helices}. 

One of the oldest questions in the study of derived and triangulated categories is the classification of full exceptional collections and semiorthogonal decompositions of a category $\cD$ up to the braid group action. To date, few results are known outside of cases where $\cD$ is indecomposable. The admissible subcategories of $\bf{P}^1$ are classified; indeed, all proper admissible subcategories are generated by one of the exceptional bundles $\cO(n)$ for $n\in \bf{Z}$. More recently, Pirozhkov \cite{pirozhkov2020admissible} showed that all semiorthogonal decompositions of $\DCoh(\bf{P}^2)$ are obtained from the Beilinson collection $\langle \cO,\cO(1),\cO(2)\rangle$ by mutation and coarsening.  

Representation theorists have also studied exceptional representations of quivers, and in some cases have achieved complete classifications of exceptional objects in triangulated categories of algebraic origin. For example, in the case of $\db(\rep\:Q)$ for $Q$ a finite acyclic quiver, Crawley-Boevey \cite{CrawleyBoevey} proves that the extended braid group action of $\bf{Z}^n\rtimes \mathfrak{B}_n$ on full exceptional collections is transitive, where $n$ is the number of vertices of $Q$. Here, the $\bf{Z}^n$ acts by shifting so that $(a_1,\ldots, a_n)\cdot (E_1,\ldots, E_n) = (E_1[a_1],\ldots, E_n[a_n])$. In \cite{Meltzerexceptional}, Meltzer proves the analogous transitivity result for the action of $\bf{Z}^n \rtimes \mathfrak{B}_n$ on exceptional objects in $\DCoh(\cX)$ for $\cX$ a weighted projective line (called a rational orbifold curve here).

However, in both the case of an acyclic quiver $\db(\rep\:Q)$ and of a weighted projective line $\DCoh(\cX)$ the classification of admissible subcategories does not follow from these results since in principle other more complicated admissible categories not admitting full exceptional collections could appear.

In the present work, we perform a thorough analysis of admissible subcategories of derived categories of so-called noncommutative curves. Examples include $\db(\rep\:Q)$ and $\DCoh(\cX)$ for $Q$ a finite acyclic quiver and for $\cX$ a smooth and proper orbifold curve -- e.g. a weighted projective line. For an Abelian category $\cA$, its \emph{homological dimension} is 
\[
    \rm{hd}(\cA) = \sup_{E,F\in \cA}\{n: \Hom_{\rm{D}(\cA)}(E,F[n])\ne 0\}.
\]
This agrees with the usual definition when $\cA$ has enough injectives or projectives. In this paper, a \emph{noncommutative curve} is a $\bf{C}$-linear Abelian category $\cA$ of homological dimension $1$ satisfying some technical conditions -- see \Cref{D:nccurve}. Subject to these conditions, noncommutative curves are classified by Reiten-van den Bergh \cite{ReitenvDB} and arise as $\rep\:Q$ for $Q$ a finite acyclic quiver or $\Coh(\cX)$ for $\cX$ a smooth and projective orbifold curve (see \Cref{T:RvDBCI}). The first theorem we prove is the following:

\begin{customthm}{A}
[=\Cref{T:mainthm}]
If $\cA$ is a noncommutative curve, then any admissible subcategory of $\db(\cA)$ is equivalent to the bounded derived category of a noncommutative curve. This implies that $\db(\cA)$ contains no phantom subcategories.
\label{T:firstthm}
\end{customthm}

We first use this result to give a complete classification of admissible subcategories of derived categories of noncommutative curves:

\begin{customthm}{B}
[=\Cref{C:acyclicquiver} + \Cref{C:rationalorbifoldindecomp} + \Cref{P:SODorbifolds}] \textcolor{white}{.}
\begin{enumerate} 
    \item If $\cA = \rep\:Q$ for $Q$ a finite acyclic quiver or $\Coh(\cX)$ for $\cX$ a smooth and proper rational orbifold curve, then every maximal length semiorthogonal decomposition of $\db(\cA)$ comes from a full exceptional collection.
    \item If $\cA = \Coh(\cX)$ for $\cX$ a smooth and proper orbifold curve with coarse moduli space $X$ of genus at least one,\footnote{Subject to the hypotheses on $\cX$, $X$ is automatically a smooth and proper curve over $\bf{C}$.} then up to mutation and coarsening every semiorthogonal decomposition is of the form
    \[
        \DCoh(\cX) = \langle \pi^*\DCoh(X),\cE\rangle
    \]
    where $\pi:\cX\to X$ is the coarse moduli map. Furthermore, $\cE$ admits a full exceptional collection of sheaves and the braid group action on $\cE$ is transitive.
\end{enumerate}
\label{T:secondthm}
\end{customthm}

Combining \Cref{T:secondthm}(1) with the results of Crawley-Boevey \cite{CrawleyBoevey} and Meltzer \cite{Meltzerexceptional} mentioned above, we obtain a complete classification of admissible subcategories of non-commutative curves. Finally, we apply these results to give the following extension of the classical Bondal-Orlov reconstruction theorem \cite{BondalOrlovrecon} -- see \Cref{S:recon} for the relevant terminology:

\begin{customthm}{C}
[=\Cref{T:BOstyle}]
    Suppose $\cX$ is a smooth and proper orbifold curve with $g(X) \ge 1$. If $\cY$ is a smooth and proper orbifold such that there is an exact equivalence $\DCoh(\cY)\simeq \DCoh(\cX)$, then $\cX \cong \cY$. 
\end{customthm}

We have also included \Cref{A:appendix} which details an elementary proof of the fact that derived categories of noncommutative curves do not admit phantom subcategories. This was the problem that motivated the present study of these examples.

\subsection*{Acknowledgements}
I thank Matthew Ballard for first encouraging me to think about non-existence of phantom categories in derived categories of acyclic quivers. I also thank Severin Barmeier, Pieter Belmans, Yuri Berest, Michel van den Bergh, and Hannah Dell for many helpful conversations. I thank Ritwick Bhargava for point out an error in an earlier version of this paper. I am especially grateful to Andres Fernandez Herrero for his continued interest in earlier versions of this paper and his helpful suggestions. Finally, it is my pleasure to thank my advisor Daniel Halpern-Leistner for his guidance and support. 

During the preparation of the current version of this paper, I became aware of the work \cite{Elaginthick} which obtains versions of \Cref{C:rationalorbifoldindecomp} and \Cref{P:SODorbifolds} among many other laudable results. Nevertheless, it is my feeling that the perspective and techniques of the present work complement the treatment of \cite{Elaginthick}. 

\subsection*{Notation and conventions}
We denote our ground field by $\bf{C}$, which may stand either for the complex numbers or any algebraically closed field of characteristic zero. All triangulated categories are assumed $\bf{C}$-linear unless otherwise stated. What we call homological dimension is also called ``global dimension'' in the theory of modules over a ring. We use this terminology because ``global dimension'' already has a meaning in the dimension theory of triangulated categories \cite{Qiugldim}. Unless otherwise specified, $\cX$ denotes a smooth and proper orbifold curve as defined in \Cref{E:stackycurve} and $\pi:\cX\to X$ is its coarse moduli space map. Such an $\cX$ is called rational if its coarse moduli space is $\bf{P}^1$.

\section{Noncommutative curves}

A $\bf{C}$-linear category $\cD$ is called \emph{finite-type} if $\Hom_{\cD}(X,Y)$ is a finite dimensional $\bf{C}$-vector space for all $X,Y\in \cD$.

\begin{defn}
    A finite-type $\bf{C}$-linear triangulated category $\cD$ is called \emph{saturated} if every (covariant and contravariant) cohomological functor $H:\cD \to \rm{mod}\:\bf{C}$ is representable. 
\end{defn}

Saturatedness implies existence of a Serre functor on $\cD$, which gives rise to a version of Serre duality \cite{B-KSerre}. By the results \textit{ibid.}, saturated triangulated categories are a natural setting in which to study admissible subcategories. 

\begin{defn} 
\label{D:nccurve}
\cite{ReitenvDB} A \emph{(smooth) noncommutative curve} over $\bf{C}$ is a $\bf{C}$-linear Noetherian Abelian category $\cA$ with $\rm{hd}(\cA) = 1$ such that $\db(\cA)$ is saturated.
\end{defn} 

Saturatedness imposes a type of categorical smoothness. In what follows, we assume all noncommutative curves are smooth in this sense.

\begin{ex}
\label{E:hereditaryalgebra}
    Let $A$ denote a finite dimensional $\bf{C}$-algebra and write $\rm{mod}\:A$ for the Abelian category of its finite dimensional modules. If $A$ is \emph{hereditary}, i.e. any submodule of a projective $A$-module is projective, then $\rm{hd}(\rm{mod}\:A) \le 1$ since every module admits a two term projective resolution. 
    
    Furthermore, it is a classical result that given a hereditary finite dimensional algebra $A$ there exists a finite acyclic quiver such that $\rep Q \cong \rm{mod}\:A$ \cite{Assem_Skowronski_Simson_2006}*{Ch. VII, Thm. 1.7}. Then, one can verify that if there are two adjacent vertices in $Q$, the corresponding simple objects admit a nontrivial extension. So, for $Q$ a connected finite acyclic quiver, $\rep Q$ is a noncommutative curve if and only if $Q$ has at least $2$ vertices. Otherwise, $\rep Q =\rm{mod}\: \bf{C}$.
\end{ex}

\begin{ex}
\label{E:stackycurve}
    Let $\cX$ denote an \emph{orbifold curve} over $\bf{C}$, by which we mean a separated irreducible finite type one-dimensional Deligne-Mumford stack with generically trivial stabilizers. A Deligne-Mumford stack $\cX$ is called \emph{projective} if it admits a closed embedding into a smooth and proper Deligne-Mumford stack with projective coarse moduli space \cite{Kresch}*{Thm. 5.3}. 
    
    Consider $\Coh(\cX)$, the category of coherent sheaves on $\cX$ where $\cX$ is a projective orbifold curve in the sense just defined. $\cX$ is Noetherian since it can be written in the form $[Y/G]$ for $Y$ a quasi-projective variety and $G$ a reductive algebraic group acting linearly on it by \cite{Kresch}. By \cite{NironiDuality}*{Thm 2.22}, $\DCoh(\cX)$ has a Serre functor given by $-\otimes_{\cO_{\cX}}\omega_{\cX}[1]$, where $\omega_{\cX}$ is the canonical sheaf of $\cX$. Then for all $F,G\in \Coh(\cX)$,
    \[
        \Ext^i_{\cO_{\cX}}(F,G) \cong \Ext^{1-i}_{\cO_{\cX}}(G,F\otimes \omega_{\cX})^\vee = 0\:\:\text{for all}\:\: i \ge 2.
    \]
In particular, $\rm{hd} \Coh(\cX) \le 1$. On the other hand, $\Ext^1(\cO_{\cX},\omega_{\cX}) = \Hom(\cO_{\cX},\cO_{\cX})^\vee \ne 0$ so that $\rm{hd} \Coh(\cX) = 1$. $\DCoh(\cX)$ is saturated by \cite{BerghLuntsSchnurer} and standard arguments show that $\Coh(\cX)$ is Noetherian. Thus, $\cX$ is a noncommutative curve. We introduce some more facts about orbifold curves in \Cref{S:admorbifolds}.
\end{ex}

In fact, \Cref{E:hereditaryalgebra} and \Cref{E:stackycurve} are \emph{all} of the ``indecomposable'' examples of noncommutative curves. We call an Abelian category $\cA$ \emph{connected} if given a decomposition $\cA = \cA_1\oplus \cA_2$ with $\cA_1$ and $\cA_2$ full Abelian subcategories of $\cA$, one has $\cA_1 = 0$ or $\cA_2 = 0$.

\begin{thm}
\label{T:RvDBCI}
\cites{ReitenvDB,ChanIngallsCoordinateRings} 
Let $\cA$ denote a connected Noetherian Abelian category with $\rm{hd}(\cA) \in \{0,1\}$ such that $\db(\cA)$ is saturated. Then $\cA$ is equivalent to  
\begin{enumerate}
    \item $\rep Q$ for $Q$ a finite acyclic quiver; or 
    \item $\Coh(\cX)$ for $\cX$ a smooth projective orbifold curve.
\end{enumerate}
Moreover, if $\rm{hd}(\cA) = 0$ then $\cA = \rm{mod}\:\bf{C}$.
\end{thm}

\begin{proof}
    By \cite{ReitenvDB}*{Thm. V.1.2}, $\cA$ is either equivalent to the category of finite dimensional modules over a hereditary algebra or to $\Coh(\cQ_X)$, where $\cQ_X$ is a sheaf of hereditary orders over a nonsingular connected projective curve $X$. In the former case, one applies \cite{Assem_Skowronski_Simson_2006}*{Ch. VII, Thm. 1.7} to obtain (1). In the latter case, one applies the results of \cite{ChanIngallsCoordinateRings} which state that $\Coh(\cQ_X)$ is equivalent to the category of coherent sheaves on an associated projective smooth orbifold curve $\cX$. Finally, suppose $\rm{hd}(\cA) = 0$. In light of the discussions in \Cref{E:hereditaryalgebra} and \Cref{E:stackycurve}, this is only possible when $\cA$ is the category of finite dimensional representations of a quiver with one vertex and no arrows, i.e. if $\cA = \rm{mod}\: \bf{C}$ as claimed.
\end{proof}

We conclude this section by characterizing the categories of type (1) in \Cref{T:RvDBCI}.

\begin{defn} 
    An Abelian category $\cA$ is called \emph{finite length} if each $A\in \cA$ is Noetherian and Artinian. It follows that each $A\in \cA$ has a filtration $0 = A_0 \subset A_1\subset \cdots \subset A_n = A$ such that 
        \begin{enumerate} 
            \item the associated graded objects $S_i = A_i/A_{i-1}$ for $1\le i \le n$ are simple; and 
            \item the set of isomorphism classes of the associated graded objects and their multiplicities are independent of the filtration.
        \end{enumerate}
    For proofs, see \cite{stacks-project}*{\href{https://stacks.math.columbia.edu/tag/0FCJ}{Tag 0FCJ}}.
\end{defn}

\begin{lem}
\label{L:finitelengthchar}
    If $\cA$ is a noncommutative curve of finite length, then $\cA$ is equivalent to the category of finite dimensional modules over a finite dimensional hereditary algebra.
\end{lem}

\begin{proof}
    Let $A$ denote a finite dimensional $\bf{C}$-algebra. One can check that $\rm{mod}\:A$ is a finite length Abelian category by arguing with dimension of submodules. By \Cref{T:RvDBCI}, it suffices to show that $\Coh(\cX)$ is never finite length for $\cX$ an orbifold curve. For this, consider $x\in \cX$ a closed point with trivial automorphism group for simplicity. There is an infinite descending chain 
    \[
        \cdots\hookrightarrow\cO_{\cX}(-nx)\hookrightarrow \cdots \hookrightarrow \cO_{\cX}(-x)\hookrightarrow \cO_{\cX} 
    \]
    of inclusions of sheaves, where $\cO_{\cX}(-x)\hookrightarrow \cO_{\cX}$ is the map arising from multiplying by a local equation cutting out $x$. $\rm{coker}(\cO_{\cX}(-x)\to \cO_{\cX})$ is a nonzero torsion sheaf and consequently this is an infinite descending chain of subobjects of $\cO_{\cX}$ which does not stabilize.
\end{proof}

\section{Admissible subcategories of noncommutative curves}

\begin{defn}\label{D:admissible}
\cite{B-KSerre} Let $\cD$ be a triangulated category and consider a full triangulated subcategory $\cB$ of $\cD$.
    \begin{itemize}
        \item $\cB^\perp$ denotes the full subcategory of $\cD$ consisting of objects $X$ such that $\Hom_{\cD}(B,X) = 0$ for all $B\in \cB$
        \item ${}^\perp \cB$ denotes the full subcategory of $\cD$ consisting of objects $X$ such that $\Hom_{\cD}(X,B) = 0$ for all $B\in \cB$.
    \end{itemize}
    We say that $\cB$ is
    \begin{enumerate} 
        \item \emph{left admissible} if every $X\in \cD$ fits into a distinguished triangle $A\to X \to B$ with $A\in {}^\perp \cB$ and $B\in \cB$;
        \item \emph{right admissible} if every $X\in \cD$ fits into a distinguished triangle $B\to X \to A$ with $B\in \cB$ and $A\in \cB^{\perp}$; and 
        \item \emph{admissible} if it is both left and right admissible.
    \end{enumerate}
\end{defn}

It is proven in \cite{B-KSerre} that for $\cD$ a $\bf{C}$-linear Hom-finite triangulated category, left and right admissibility are equivalent in the presence of a Serre functor. These hypotheses will hold in all of the examples considered here.

\begin{lem}
\label{L:admissiblethick}
    If $\cB$ is an admissible subcategory of $\cD$ then $X\oplus Y \in \cB$ implies $X\in \cB$.
\end{lem}

\begin{proof}
    This follows from $\cB = {}^\perp(\cB^\perp)$; see \cite{stacks-project}*{\href{https://stacks.math.columbia.edu/tag/0CQP}{Section 0CQP}}.
\end{proof}

For an Abelian category $\cA$, its bounded derived category $\db(\cA)$ has a standard bounded t-structure with t-cohomology functors denoted $\{H^i\}_{i\in \bf{Z}}$. In particular, for any $X\in \db(\cA)$, $H^i(X) \in \cA$ for all $i$. We always regard $\cA$ as the full subcategory of $\db(\cA)$ whose objects have cohomology only in degree $0$.

\begin{lem}
\label{L:hd1lem}
    If $\cA$ is an Abelian category with $\rm{hd}(\cA) \in \{0,1\}$, then for every object $X\in \db(\cA)$ one has $X \cong \bigoplus_{i\in \bf{Z}}H^i(X)[-i]$, where all but finitely many summands are zero.
\end{lem}

\begin{proof}
    This is well known, see \cite{Bruning}*{Lem. 3.3} for instance.
\end{proof}

\begin{lem}
\label{L:intersectheart}
    Suppose $\cA$ is an Abelian category with $\rm{hd}(\cA) \in\{0,1\}$ and consider an admissible subcategory $\cB\subset \db(\cA)$. $\cA\cap \cB$ is the heart of a bounded t-structure on $\cB$. 
\end{lem}

\begin{proof}
    Consider $X\in \cB$ and write $X = \bigoplus_{k\in \bf{Z}} H^k(X)[-k]$ as in \Cref{L:hd1lem}. $\cB$ is closed under summands and shifts, so $H^k(X)\in \cB$ for all $k$. On the other hand, associated to the standard t-structure is a filtration of $0 = X_{n+1}\to X_n \to \cdots \to X_2\to X_1 = X$ such that $\Cone(X_{i+1}\to X_i)$ is given by one of the nonzero $H^{k_i}(X)[-k_i]$. It follows that each $X_i$ lies in $\cB$ and consequently that $\cA \cap \cB$ is the heart of a bounded t-structure on $\cB$.
\end{proof}

\begin{prop}
\label{P:inducedtstructure}
    Suppose $\cA$ and $\cB$ are as in \Cref{L:intersectheart}. There exists an exact equivalence $r:\db(\cA\cap \cB)\to \cB$ which restricts to the identity on $\cA\cap \cB$.
\end{prop}

\begin{proof}
    By \cite{LPZ}*{Lem. 3.1}, there is a t-exact functor $r:\db(\cA\cap \cB)\to \cB$ which restricts to the identity on the heart $\cA\cap \cB\subset \cA$ of \Cref{L:intersectheart}. By \cite{LPZ}*{Lem. 3.3}, this is an exact equivalence, since given a pair of objects $X,Y$ of $\cA \cap \cB$, one has $\Hom_{\db(\cA)}(X,Y[n]) = 0$ for all $n\not\in \{0,1\}$. 
\end{proof}

\begin{thm}
\label{T:mainthm}
    If $\cA$ is a noncommutative curve, then any admissible subcategory of $\db(\cA)$ is equivalent to the bounded derived category of a category of type (1) or (2) in \Cref{T:RvDBCI}. In particular, $\db(\cA)$ contains no phantom subcategories.
\end{thm}

\begin{proof}
    Let $\cB$ be an admissible subcategory of $\db(\cA)$ and consider $X,Y$ in the associated heart $\cA \cap \cB$ by \Cref{L:intersectheart}. By \Cref{P:inducedtstructure} one sees $r(X) = X$ and $r(Y[n]) = Y[n]$. Consequently, $\Hom_{\db(\cA\cap \cB)}(X,Y[n]) = \Hom_{\db(\cA)}(X,Y[n])$ for all $n\in \bf{Z}$. It follows that $\rm{hd}(\cA\cap \cB) \in \{0,1\}$. $\cA\cap \cB$ is Noetherian, being a full Abelian subcategory of $\cA$. Finally, saturatedness is a property inherited by semiorthogonal factors \cite{B-KSerre}*{Prop. 2.8} so by \Cref{P:inducedtstructure} we conclude that $\cA\cap \cB$ is of the claimed form.
    
    It remains to observe that each of the categories listed in \Cref{T:RvDBCI} has nontrivial $\rm{K}_0$. This follows from the existence of $\dim: \rm{K}_0(\rep Q)\to \bf{Z}$ and $\rm{rank}:\rm{K}_0(\Coh(\cX)) \to \bf{Z}$, both of which are nonzero. Thus, $\rm{K}_0(\cB)$ has positive rank and $\cB$ is not a phantom category.
\end{proof}


\begin{cor}
\label{C:acyclicquiver}
    Let $Q$ denote a finite acyclic quiver. If $\cB \subset \db(\rep Q)$ is a nontrivial admissible subcategory then $\cB \simeq \db(\rep Q')$ for $Q'$ a finite acyclic quiver.
\end{cor}

\begin{proof}
    By the proof of \Cref{T:mainthm}, $\cB$ is equivalent as triangulated categories to $\db(\rep Q\cap \cB)$, where $\rep Q\cap \cB$ is one of the categories classified in \Cref{T:RvDBCI}. As an Abelian subcategory of $\rep Q$, $\rep Q \cap \cB$ is also of finite length and thus by \Cref{L:finitelengthchar} equivalent to $\db(\rep Q')$ for $Q'$ another finite acyclic quiver.
\end{proof}

If $Q$ is a finite acylic quiver with $n$ vertices, then the simple representations give a full exceptional collection of $\db(\rep Q)$. On the other hand, work of Crawley-Boevey \cite{CrawleyBoevey} implies that the extended mutation action by $\bf{Z}^n \rtimes \mathfrak{B}_n$ on the full exceptional collections in $\db(\rep Q)$ is transitive. \Cref{C:acyclicquiver} implies that every admissible subcategory of $\db(\rep Q)$ admits a full exceptional collection and in particular we obtain a complete classification of the admissible subcategories of $\db(\rep Q)$.

\Cref{T:mainthm} also allows us to classify admissible subcategories of orbifold curves $\cX$ which are \emph{rational} in the sense that $X = \bf{P}^1$. 

\begin{cor}
\label{C:rationalorbifoldindecomp}
    If $\cX$ is rational, then any indecomposable admissible subcategory of $\DCoh(\cX)$ is generated by an exceptional object.
\end{cor}

\begin{proof}
    Given an admissible subcategory $\cA$ of $\DCoh(\cX)$, it follows that $\rm{K}_0(\cA)$ is a summand of $\mathrm{K}_0(\cX)$. As is well-known, $\mathrm{K}_0(\cX)$ is of finite rank (see \Cref{P:SODorbifolds} for example) and thus it follows that if $\cA$ cannot be equivalent to $\DCoh(X)$ for $X$ a curve of positive genus; if not, $\Pic^0(X)$, which is of infinite rank, would be contained as a summand of $\mathrm{K}_0(\cX)$, 
\end{proof}

In particular, every semiorthogonal decomposition of $\DCoh(\cX)$ for $\cX$ rational is refined by a full exceptional collection. The systematic study of derived categories of rational orbifold curves, otherwise called weighted projective lines, is by now classical and was initiated by representation theorists. Consequently, the structure of helices in these categories is well understood. Indeed, \cite{Meltzerexceptional}*{Thm. 3.3.1} states that the braid group action $\mathfrak{B}_n$ on full exceptional sequences in $\Coh(\cX)$ is transitive. Consequently, the induced action by $\bf{Z}^n \rtimes \mathfrak{B}_n$ on full exceptional collections in $\DCoh(\cX)$ is transitive \cite{Meltzerexceptional}*{Cor. 3.3.2}. Combining this with \Cref{C:rationalorbifoldindecomp} we obtain a complete description of the admissible subcategories of $\DCoh(\cX)$ for $\cX$ rational.



\section{Admissible subcategories of orbifold curves}
\label{S:admorbifolds}
In this section, we study the admissible subcategories of smooth and proper orbifold curves as introduced in \Cref{E:stackycurve}. We recall here several relevant facts about orbifold curves that we will use below.

It follows from the definitions that the coarse moduli space $X$ of a smooth and proper orbifold curve $\cX$ is a smooth and proper curve \cite{Taams}*{Lem. 1.1.8}. This implies that the coarse moduli map $\pi:\cX\to X$ is flat and in particular induces an exact functor $\pi^*:\Coh(X)\to \Coh(\cX)$. On the other hand, since the characteristic of our ground field $\bf{C}$ is zero, $\pi$ is tame and $\pi_*:\Coh(\cX)\to \Coh(X)$ is exact \cite{Alper}*{\S 4.4}. In this setting one can also show that the unit of adjunction
\[
    \eta :\id \Rightarrow \pi_*\pi^*
\]
is an isomorphism.

A smooth and proper orbifold curve $\cX$ is uniquely determined by its coarse moduli space $X$ and a finite collection of points $p_1,\ldots, p_n \in X$ with weights $e_1,\ldots, e_n \in \bf{Z}_{\ge 2}$. Indeed, $\cX$ can be obtained as an iterated root stack over the divisors $p_i$, where we let $\cX^0 := X$ and put $\cX^i := \sqrt[e_i]{X/p_i}$ for $i=1,\ldots, n$. It can be shown that $\cX^n \cong \cX$ -- see \cite{Behrend_2014} for details. 

Given a closed point $q\in \cX$, we write $e_q$ for the order of its automorphism group, which is always cyclic; note also that for all but finitely many $q$ one has $e_q = 1$.

The universal property of the root stack construction implies that $\cX$ has a line bundle $L_i = \cO_{\cX}(\tfrac{1}{e_i}\cdot p)$ for each $i = 1,\ldots,n$ with the property that $L_i^{\otimes e_i} \cong \pi^*\cO_{X}(p)$. In this sense, the root stack construction adjoins an $e_i^{\rm{th}}$ root of the line bundle $\cO_X(p_i)$ for all $i=1,\ldots, n$. The Picard group $\Pic(\cX)$ is 
\[
    \Pic(\cX) = \pi^*\Pic(X)[L_1,\ldots,L_n]
\]
subject to the relations $L_i^{\otimes e_i} = \pi^*\cO_X(p_i)$ for $i=1,\ldots,n$ --  see \cite{Behrend_2014}*{p. 122}. It can also be verified that 
\begin{equation}
\label{E:canonicalclass}
    \omega_{\cX} \cong \pi^*\omega_X \otimes \cO_{\cX}\left(\sum_{i=1}^n \frac{e_i-1}{e_i}\cdot p_i\right).
\end{equation}
The reader is encouraged to consult \cite{BGRdimension}*{\S 2} and \cites{Behrend_2014,Taams} for more about orbifold curves and their derived categories. We include several results here that we will use in the sequel. 

First, it is crucial to understand the structure of the category of sheaves on $\cX$ supported at the stacky points of $\cX$. Suppose that $p$ is a stacky point of $\cX$ with order $e_p = n$. We let $\cU_p$ denote the category of coherent sheaves supported at $p$. Here, support is defined as in \cite{Nironi09}. The category $\cU_p$ is equivalent to the category $\rep_0(Z_n)$ of nilpotent representations of a cyclic quiver $Z_n$ with $n$ vertices:
\[
    \begin{tikzcd}
        0\arrow[r] & 1 \arrow[r]& \cdots  \arrow[r]& n-1 \arrow[lll,bend left]
    \end{tikzcd}
\]

Representations of $Z_n$ correspond to $\bf{Z}/n\bf{Z}$-graded vector spaces $V = \bigoplus_{k\in \bf{Z}/n\bf{Z}}V_k$, equip\-ped with a degree $1$ endomorphism $T$. The simple objects of $\cU_p$ are those with exactly one $V_k$ nonzero and isomorphic to $\bf{C}$. We index these as $S_0$ for $k=0,\ldots, n-1$. 

The category $\cU_p$ has an Auslander-Reiten translation functor $\tau$ which amounts reindexing a representation by counterclockwise rotation. More precisely, it takes the corresponding graded vector space $V = \bigoplus_{k \in \bf{Z}/n\bf{Z}} V_k$ to $V(1) = \bigoplus_{k\in \bf{Z}/n\bf{Z}} V_k(1)$ where $V_k(1) = V_{k+1}$ and adjusts the endomorphism accordingly. It follows from this description that $\tau^n = \id$.

\begin{ex} 
    The indecomposable modules can also be explicitly described. We content ourselves to describe the ones of length at most $n$. Given $1\le l \le n-1$, we construct an indecomposable nilpotent representation of $Z_n$, denoted $M_l$, by putting $V_1 = V_2 = \cdots = V_l = \bf{C}$ and all other $V_k = 0$ and letting $V_{i}\to V_{i+1}$ be the identity for $i=0,\ldots, l-1$. We define $M_n$ by putting $V_0 = \cdots = V_{n-1} = \bf{C}$, letting all morphisms $V_i \to V_{i+1}$ be the identity for $i=0,\ldots,n-2$, and letting $V_{n-1}\to V_0$ be the zero morphism. The other indecomposable objects in $\cU_p$ of length $1\le l\le n$ can be obtained from $M_l$ by repeated application of $\tau$.
\end{ex}

As a special case of the above discussion, the subcategory of $\Coh(X)$ consisting of sheaves supported at $\pi(p)$ is equivalent to the category $\rep_0(Z_1)$, i.e. the category of finite dimensional nilpotent representations of the Jordan quiver. The exact functor $\pi_*:\Coh(\cX) \to \Coh(X)$ thus induces a functor $\rep_0(Z_n) \to \rep_0(Z_1)$. 

\begin{lem}
\label{L:pfquiver}
    The functor $\rep_0(Z_n) \to \rep_0(Z_1)$ induced by $\pi_*$ sends $(V = \bigoplus_{k\in \bf{Z}/n\bf{Z}} V_k,T)$ to the vector space $V_0$ equipped with the endomorphism $T^n|_{V_0}$.
\end{lem}

\begin{proof}
    We will not need the full strength of the lemma, so we only sketch its proof. As is proven in \cite{Taams}*{Cor. 1.1.31}, there is a strong local normal form for coarse moduli maps of stacky curves, which is given by the Cartesian diagram of stacks:
    \[
        \begin{tikzcd}
            \left[W/\mu_n\right] \arrow[d]\arrow[r]&\cX\arrow[d,"\pi"]\\
            W/\mu_n \arrow[r]&X
        \end{tikzcd}
    \]
    where the vertical arrows are coarse moduli maps, the bottom horizontal arrow is a Zariski open immersion, and $W$ is a smooth affine curve with a $\mu_n$ action with fixed point $q$. We also denote by $q$ the corresponding points in $[W/\mu_n]$ and $\cX$, and by $\pi(q) = p$ the corresponding point in $X$. It follows that the category $\cU_q$ can be identified with the category of $\mu_n$-equivariant sheaves supported at $q\in W$, or equivalently the category of $\bf{Z}/n\bf{Z}$-graded $\cO_{V,q}$-modules $V = \bigoplus_{k\in \bf{Z}/n\bf{Z}} V_k$. Denote by $x$ a uniformizer of $\cO_{X,p}$ and by $u$ a uniformizer of $\cO_{V,q}$. $u$ acts as the degree $1$ endomorphism of $M$ that we have seen above and the morphism $f:W\to W/\mu_n$ gives an equality $f^*(x) = u^n$ up to a scalar.
    
    Subject to these identifications, $\pi_*$ induces a functor from the category of $\mu_n$-equivariant torsion sheaves supported at $q\in W$ to the category of torsion sheaves supported at $p\in X$ by sending $(V,u) \mapsto (V_0,x)$, where we note that $f^*(x) = u^n$ allows us to identify $x$ with an endomorphism of $V_0$. 
\end{proof}

\begin{cor}
\label{C:pushforwardtorsionsheaf}
    For any $1\le l \le n$, $\pi_*$ sends the coherent sheaf on $\Coh(\cX)$ corresponding to $M_l$ to $\bf{C}_p$.
\end{cor}

\begin{proof}
    This is an immediate consequence of \Cref{L:pfquiver} combined with the observation that $\bf{C}_p$ corresponds to the dimension $1$ representation of $Z_1$ with zero endomorphism. 
\end{proof}

\begin{prop}
\label{P:stackySOD}
    There is a semiorthogonal decomposition $\DCoh(\cX) = \langle \pi^*\DCoh(X), \cE\rangle$, where $\cE$ has a full exceptional collection of sheaves supported at the stacky points of $\cX$.
\end{prop}

\begin{proof}
    This is proved in far greater generality in \cite{Bodzentadonovan}, but we give an argument here in this simple special case. $\pi^*:\DCoh(X) \to \DCoh(\cX)$ is fully faithful, with right adjoint given by $\pi_*$. Since both the source and target categories have Serre functors, denoted $S_X$ and $S_{\cX}$ respectively, a left adjoint $\pi_!\dashv \pi^*$ can be constructed by $\pi_! = S_X^{-1} \circ \pi_* \circ S_{\cX}$. However, an application of the projection formula shows that $\pi_! \simeq \pi_*$. As a consequence, by \cite{stacks-project}*{\href{https://stacks.math.columbia.edu/tag/0CQT}{Tag 0CQT}} there is a semiorthogonal decomposition $\DCoh(\cX) = \langle \pi^*\DCoh(X),\cE\rangle$, where $\cE$ is defined as the complement to $\pi^*\DCoh(X)$.
    
    Next, $\cE$ can be characterized as the kernel of the functor $\pi_*$ and it follows from \Cref{L:intersectheart} that $\cE$ is generated as a triangulated category by its intersection with $\Coh(\cX)$. An object sent to zero by $\pi_*$ must be supported at the points $p_1,\ldots, p_n$ so it suffices to consider what objects in $\cU_{p_i}$ are sent to zero under $\pi_*$ for $i=1,\ldots, n$. \Cref{L:pfquiver} implies that $\ker(\pi_*) \cap \cU_{p_i}$ is generated by the simples $S_1,\ldots, S_{e_i-1} \in \cU_{p_i}$, each of which is exceptional. It follows that the triangulated subcategory $\cE_{p_i}$ generated by $\ker(\pi_*)\cap \cU_{p_i}$ has a full exceptional collection $\langle S_1,\ldots, S_{e_i-1}\rangle$. By disjointness of support, we obtain an orthogonal decomposition $\cE = \cE_{p_1}\oplus \cdots \oplus \cE_{p_n}$ and the result is now proven.
\end{proof}

\begin{rem}
\label{R:An}
    More is true in fact, the objects $S_1,\ldots, S_{e_i-1}$ generate a subcategory of $\Coh(\cX)$ isomorphic to $\rep(A_{e_i-1})$ and in particular $\db(\rep\:A_{e_i-1}) \simeq \cE_{p_i}$ for $i=1,\ldots, n$. In particular, by \cite{CrawleyBoevey} $\bf{Z}^n\rtimes \mathfrak{B}_{e_{i}-1}$ acts transitively on the full exceptional collections of $\cE_{p_i}$ for each $i=1,\ldots, n$.
\end{rem}

\begin{lem}
\label{L:francoslemma}
    Suppose $g(X) \ge 1$ and that $E\in \Coh(\cX)$ fits into a triangle
    \begin{equation}
\label{E:triangle}
    B\to E \xrightarrow{f} A \to B[1]
\end{equation}
such that $\Hom(B,A[i]) = 0$ for all $i\in \bf{Z}$.
    \begin{enumerate}
        \item If $E = \pi^*\bf{C}_q$ for $q\ne p$, then $E$ is a summand of $B$ or $A$.\vspace{2mm}
        \item If $E$ is a line bundle, then either $B \to E$ is zero or $B$ has a line bundle as a summand.
    \end{enumerate}
\end{lem}

\begin{proof}
    Consider the following portion of the long exact sequence of cohomology objects with respect to the standard $t$-structure on $\DCoh(\cX)$:
    \[
        0 \to H^{-1}(A) \to H^0(B) \to E\to H^0(A) \to H^1(B)\to 0.
    \]
    We consider (1) where $E = \pi^*\bf{C}_q$. The map $H^0(B) \to \pi^*\bf{C}_q$ is either surjective or zero. Suppose first that it is surjective so we have 
    \[
        0\to H^{-1}(A) \xrightarrow{h} H^0(B) \to \pi^*\bf{C}_q\to 0.
    \]
    Suppose $h\ne 0$. We may assume that the torsion parts of $H^{-1}(A)$ and $H^0(B)$ away from $q$ are trivial since if not, by \Cref{L:hd1lem}, $h$ would induce an isomorphism between a summand of $A[-1]$ and a summand of $B$, contrary to the Hom-vanishing assumption. So, we may assume that the cohomology objects have torsion only at $q$. Note that by \cite{BGRdimension}*{Lem. 2.11} each $F\in \Coh(\cX)$ splits as $F = F_{\rm{tf}}\oplus F_{\rm{tor}}$ into its torsion free and torsion parts. If $h$ induces a non-trivial map $H^{-1}(A)_{\rm{tor}}\to H^0(B)_{\rm{tor}}$ then 
    \[
        \Hom(H^{-1}(A)_{\rm{tor}},H^0(B)_{\rm{tor}}) \cong \Hom(H^{-1}(A)_{\rm{tor}},H^0(B)_{\rm{tor}}\otimes \omega_{\cX}) \ne 0.
    \]
    It follows that 
    \begin{equation}
    \label{E:SDcontradiction}
        0 \ne \Hom(H^{-1}(A),H^0(B)\otimes \omega_{\cX}) \cong \Hom(H^0(B)[-1], H^{-1}(A))^\vee
    \end{equation}
    by Serre duality and we obtain a contradiction. So, we may suppose that $H^{-1}(A)_{\rm{tor}} = 0$. If $H^0(B)_{\rm{tor}} \ne 0$ then $H^0(B)_{\rm{tor}} \cong \pi^*\bf{C}_q$ and thus $\pi^*\bf{C}_q$ is a summand of $B$ as desired. So, we may assume that $H^{-1}(A)$ and $H^0(B)$ are torsion free and hence locally free. In this case, if $h\ne 0$ then $\Hom(H^{-1}(A),H^0(B)) \ne 0$ and since $\omega_{\cX}$ has global sections it follows that $\Hom(H^{-1}(A),H^0(B)\otimes \omega_{\cX}) \ne 0$. However, duality as in \eqref{E:SDcontradiction} leads to a contradiction.
    
    Suppose next that $H^0(B) \to \pi^*\bf{C}_q$ is zero. We obtain a short exact sequence 
    \[
        0 \to \pi^*\bf{C}_q \to H^0(A)\to H^1(B)\to 0
    \]
    where we may again assume that $H^0(A)$ and $H^1(B)$ have torsion only at $q$. Suppose $\pi^*\bf{C}_q\to H^0(A)_{\rm{tor}}$ is not an isomorphism, since otherwise $\pi^*\bf{C}_q$ is a summand of $A$ and we are done. Then, there is a non-zero morphism $H^0(A)_{\rm{tor}} \to H^1(B)_{\rm{tor}}$ and applying the same argument above using Serre duality we obtain a contradiction. 

    For (2), since $H^0(B)_{\rm{tor}}\to L$ is zero, it follows that $H^{-1}(A)_{\rm{tor}}\to H^0(B)_{\rm{tor}}$ is an isomorphism. Consequently, we may assume that $H^{-1}(A)_{\rm{tor}} = H^0(B)_{\rm{tor}} = 0$. If $H^{-1}(A) \to H^0(B)$ is non-zero, then since these sheaves are locally free sheaves we obtain a non-zero morphism $H^{-1}(A)\to H^0(B)\otimes \omega_{\cX}$ and apply \eqref{E:SDcontradiction} to derived a contradiction. So, $H^{-1}(A) = 0$ and we have $H^0(B) \hookrightarrow L$. So, either $H^0(B)$ is a line bundle or it is zero. We may assume we are in the latter case, and thus we have 
    \[
        0 \to L \to H^0(A) \to H^1(B) \to 0
    \]
    Again, we may assume that $H^0(A)$ and $H^1(B)$ are locally free. If $H^0(A) \to H^1(B)$ is non-zero we obtain the same contradiction as before, so it follows that $L \to H^0(A)$ is an isomorphism and that $H^1(B) = 0$. 

    By \Cref{L:hd1lem}, we have that $B = \bigoplus_{i\in \bf{Z}}H^i(B)[-i]$ and since $\Coh(\cX)$ is of homological dimension $1$, it follows that $\Hom(B,E) = \Hom(H^1(B)[-1],E) \oplus \Hom(H^0(B),E) = 0$.
\end{proof}

We will also need the following basic lemma:

\begin{lem}
\label{L:generatorX}
    Let $q\in X$ and a line bundle $L$ be given. Then, $\bf{C}_q\oplus L$ is a classical generator of $\DCoh(X)$.
\end{lem}

\begin{proof}
    Embed $X \hookrightarrow \bf{P}^N$ using a very ample line bundle of the form $\cO_X(\ell\cdot q)$ and apply \cite{stacks-project}*{\href{https://stacks.math.columbia.edu/tag/0BQT}{Tag 0BQT}} to see that $\cO_X\oplus \cO_X(-\ell \cdot q)\oplus \cdots \oplus \cO_X(-\ell N\cdot  q)$ is a generator of $\Dqc(X)$ and hence $\DCoh(X)$. Now, applying the autoequivalence $(-)\otimes_{\cO_X} L$ of $\DCoh(X)$ shows that $L\oplus L(-\ell \cdot q)\oplus \cdots \oplus L(-\ell N\cdot q)$ is a generator for any $L$. However, since $\fib(L \to \bf{C}_q) = L(-q)$ we see that the subcategory generated by $\bf{C}_q$ and $L$ contains $L(-nq)$ for all $n\ge 0$.
\end{proof}

\begin{prop}
\label{P:SODorbifolds}
    Suppose that $g(X) \ge 1$. Given a semiorthogonal decomposition $\DCoh(\cX) = \langle \cA_1,\cA_2\rangle$, up to mutation and tensoring by a line bundle of the form 
    \begin{equation}
    \label{E:linebundle}
        \cO_{\cX}\left(\sum_{i=1}^n \frac{k_i}{e_i}\cdot p_i\right) \text{ with } 0\le k_i<e_i \text{ for all } i=1,\ldots, n
    \end{equation}
    we have that $\cA_1$ contains $\pi^*\DCoh(X)$ as a full subcategory.
\end{prop}

\begin{proof}
    Let $E$ be an object of $\Coh(\cX)$ and consider the canonical triangle $B \to E \to A$ coming from $\DCoh(\cX) = \langle \cA_1,\cA_2\rangle$. By semiorthogonality, $\Hom(B,A[i]) = 0$ for all $i\in \bf{Z}$ and we may apply \Cref{L:francoslemma}. If $E = \pi^*\bf{C}_q$ for $q \not \in \{p_1,\ldots, p_n\}$, then either $\cA_1$ or $\cA_2$ contains $\pi^*\bf{C}_q$. Up to mutating, we may assume that $\pi^*\bf{C}_q \in \cA_1$ and thus that there are no locally free sheaves in $\cA_2$ by semiorthogonality. However, taking a line bundle $L$ and decomposing it as $B\to L \to A$ like above, we apply \Cref{L:francoslemma} to see that $B\to L$ is zero and thus that $L \in \cA_1$. 
    
    After twisting $\cA_1$ by a line bundle as in \eqref{E:linebundle}, we may assume that $L = \pi^*(L')$ for $L'$ a line bundle on $X$ by the description of $\Pic(\cX)$ in \cite{Taams}*{Cor. 1.2.12}. Next, since $\pi^*:\DCoh(X) \to \DCoh(\cX)$ is fully faithful, it follows that $\cA_1$ contains the essential image of the category generated by $L'$ and $\bf{C}_q$, which by \Cref{L:generatorX} implies that $\cA_1$ contains $\pi^*\DCoh(X)$. 
\end{proof}

\begin{rem}
    Note that \Cref{P:SODorbifolds} contains as a special case Okawa's indecomposability result \cite{Okawaindecomp} for derived categories of smooth and proper curves $X$ of genus at least one. Indeed, if $\cX = X$ is a classical curve of genus at least one, then $\pi = \id$ and $\cA_1 = \DCoh(X)$.
\end{rem} 

\Cref{P:SODorbifolds} combined with \Cref{R:An} gives a satisfactory answer to the question of classification of semiorthogonal decompositions in the case where $g(X) \ge 1$. We will conclude this section by classifying admissible subcategories of $\DCoh(\cX)$ equivalent to $\DCoh(\bf{P}^1)$ in the case where $\cX$ is rational. We first record the following lemma:

\begin{lem}
\label{L:embeddingexact}
    Consider an admissible embedding $i:\DCoh(\bf{P}^1)\hookrightarrow \DCoh(\cX)$. It follows that $i$ embeds $\Coh(\bf{P}^1)$ into $\Coh(\cX)[n]$ for some $n\in \bf{Z}$ and that $g(X) = 0$.
\end{lem}

\begin{proof}
    By \cite{HLJR}*{Prop. 3.1}, $i$ takes objects of $\Coh(\bf{P}^1)$ into a uniformly bounded set of cohomological degrees $[a,b]\cap \bf{Z}$ for $a<b \in \bf{Z}$. Since any line bundle $L$ in $\Coh(\bf{P}^1)$ has $\Hom(L,L)\cong \bf{C}$, it follows that $i(L)$ must be in a shift of $\Coh(\cX)$ by \Cref{L:hd1lem}. Consequently, there are two line bundles $\cO(k)$ and $\cO(n)$ whose images under $i$ lie in the same shift of $\Coh(\cX)$. Thus, by the exact triangles in \cite{GKR}*{Lem. 6.6} it follows that $i$ takes $\Coh(\bf{P}^1)$ to this shift of $\Coh(\cX)$.

    We omit the proof of the second claim, since we will not use it. We note, however, that it follows from the classification of exceptional pairs in $\Coh(\cX)$ for $g(X) \ge 1$. The key fact is that there are no exceptional bundles on $\cX$ -- see \cite{Elaginthick}*{Cor. 8.5}.
\end{proof}

\begin{rem}
    Note that one cannot directly apply \Cref{P:inducedtstructure}, since in this case of $\DCoh(\bf{P}^1) \hookrightarrow \DCoh(\cX)$ it could be the case \emph{a priori} that $\Coh(\cX)$ restricts to the heart $\rep(K_2)$, where $K_2$ is the Kronecker quiver.
\end{rem}

\begin{prop}
\label{P:p1inwhengenus0}
    If $g(X) = 0$, then all admissible embeddings $\DCoh(\bf{P}^1)\hookrightarrow\DCoh(\cX)$ have essential image of the form $\pi^*\DCoh(\bf{P}^1)\otimes L$ for $L$ as in \eqref{E:linebundle}. 
\end{prop}

\begin{proof}
    By \Cref{L:embeddingexact}, we may assume that $i$ induces an embedding $\Coh(\bf{P}^1)\hookrightarrow \Coh(\cX)$ so that $i(\cO)$ and $i(\cO(1))$ form an exceptional pair in $\Coh(\cX)$. Exceptional pairs in $\Coh(\cX)$ are well-studied in the literature, and by \cite{LenzingMeltzer}*{Thm. 3} and its proof it follows that $i(\cO) = \pi^*\cO\otimes L'$ and $i(\cO(1)) = \pi^*\cO(1)\otimes L'$ for $L' \in \Pic(\cX)$. Noting that every $L'\in \Pic(\cX)$ is of the form $L' = \pi^*\cO(n) \otimes L$ for $L$ as in \eqref{E:linebundle}, the result follows.
\end{proof}

\section{Reconstruction theorem}
\label{S:recon}

As a corollary of the previous work, we will explain how an orbifold curve $\cX$ can be recovered from the triangulated category $\DCoh(\cX)$. The objective of this section is to prove the following extension of the Bondal-Orlov reconstruction theorem \cite{BondalOrlovrecon} to the case of orbifold curves. By an \emph{orbifold}, we mean a separated Deligne-Mumford stack with generically trivial stabilizer over an algebraically closed ground field $\bf{C}$ of characteristic zero  -- see \cite{Kresch}*{\S 4.1}.

\begin{thm}
\label{T:BOstyle}
    Suppose $\cX$ is an orbifold curve with $g(X) \ge 1$. If $\cY$ is a smooth and proper orbifold such that there is an exact equivalence $\DCoh(\cY)\simeq \DCoh(\cX)$, then $\cX \cong \cY$. 
\end{thm}

We will first prove some lemmas that will be necessary later. The notion of \emph{point-like} (or simply ``point'') objects was introduced in \cite{BondalOrlovrecon}*{Def. 2.1}. 

\begin{lem}
\label{L:pointlike}
    Suppose that $g(X)\ge 2$ or that $g(X) = 1$ and $\cX$ has at least one stacky point. Then, the point-like objects of $\DCoh(\cX)$ are of the form $\pi^*\bf{C}_q[k]$ where $q\ne p_1,\ldots, p_n$ and $k\in\bf{Z}$.
\end{lem}

\begin{proof}
    Subject to the above hypotheses, if $g(X) \ge 1$ then $\omega_{\cX}$ is ample, in the sense that there exists $k \ge 0$ such that $\omega_{\cX}^{\otimes k} \cong \pi^*L$ for $L$ an ample line bundle on $X$. This can be deduced from \eqref{E:canonicalclass} by taking $k = \rm{lcm}(e_1,\ldots, e_n)$, for example. Now, if $P$ is a point-like object of $\DCoh(\cX)$, then it is indecomposable and hence we may assume that it is in $\Coh(\cX)$. Then, since $P\otimes \omega_{\cX}^{\otimes k} \cong P$ for all $k\ge 0$ we see that $P \otimes \pi^*L \cong P$ for some ample $L$. The projection formula implies that $\pi_*(P) \cong \pi_*(P) \otimes L$. Then, by \cite{HuybrechtsFM}*{p. 93} it follows that the dimension of the support of $\pi_*P$, and hence of that of $P$, is zero.

    So, say that $\pi_*P$ is supported at some point $q\in X$, which is unique by indecomposability. Then, $P$ corresponds to an indecomposable nilpotent representation of the cyclic quiver with $e_q$ vertices. In the case where $e_q=1$, i.e. when $q \ne p_1,\ldots, p_n$, we see that $P = \pi^*\bf{C}_q[k]$ by the condition that $\Hom(P,P) \cong \bf{C}$. Otherwise, we note that $\omega_{\cX}$ restricted to the category $\cU(p_i)$ acts nontrivially as a power of the Auslander-Reiten translation, $\tau^{{e_i}-1}$. In particular, by the classification of the indecomposable objects in $\cU_{p_i}$ it follows that $P\otimes \omega_{\cX} \cong P$ is not possible.
\end{proof}

\Cref{L:pointlike} can also be extended to cover certain cases where $X = \bf{P}^1$, but we will not have need of this.

\begin{proof}[Proof of \Cref{T:BOstyle}]
    First, note that the dimension of $\cX$ can be recovered from $\DCoh(\cX)$ using, for example, Serre dimension by \cite{BGRdimension}*{Cor. 5.2}. Consequently, if $\cY$ is a proper orbifold such that $\DCoh(\cY) \simeq \DCoh(\cX)$ then the coarse space $Y$ must be a smooth and proper curve. Thus, by \cite{Behrend_2014}*{Thm 3.71} it follows that $\cY$ is a gerbe over an orbifold curve. The generically trivial stabilizer hypothesis on $\cY$ implies that it is an orbifold curve. It now remains to show that given a category $\cD$ equivalent to the bounded derived category of an orbifold curve $\cX$, one can determine the curve $\cX$ up to isomorphism from $\cD$.
    
    First, note that it can be determined whether or not $\cX$ is rational by considering whether or not $\rm{K}_0(\cD)$ is of finite rank. Now, if $\cD$ is indecomposable, then by \Cref{P:stackySOD} it follows that $\cX = X$ with $g(X)\ge 1$ and the result follows from \cite{HuybrechtsFM}*{Cor. 5.46}. Otherwise, we define an admissible category $\cA$ of $\cD$ as follows:

    \begin{enumerate}
        \item if $g(X) \ge 1$, then let $\cA$ be a proper indecomposable admissible subcategory which contains no exceptional object; and \vspace{2mm}
        \item if $g(X) = 0$, let $\cA$ be any admissible subcategory equivalent to $\DCoh(\bf{P}^1)$.
    \end{enumerate}
    It now follows that choosing an equivalence $f:\cD \to \DCoh(\cX)$, we have a commutative diagram:
    \[
        \begin{tikzcd}
            \cD \arrow[r,"f"]& \DCoh(\cX)\\
            \cA \arrow[u,hook,"i"] \arrow[r,"\sim"] & \pi^*\DCoh(X) \otimes L\arrow[u,hook] & \DCoh(X) \arrow[l,"\sim",swap]\arrow[ul,dotted,"j"]
        \end{tikzcd}
    \]
    where $L$ is as in \eqref{E:linebundle}, by the proof of \Cref{P:SODorbifolds} when $g(X)\ge 1$ and by \Cref{P:p1inwhengenus0} otherwise. Now, the functor $\pi_*((-)\otimes L^\vee): \DCoh(\cX) \to \DCoh(X)$ can be characterized as the right adjoint to the inclusion $j$ and so there is a 2-commutative diagram:
    \[
        \begin{tikzcd}
            \cD \arrow[r,"f"] \arrow[d,"r",swap]& \DCoh(\cX)\arrow[d,"\pi_*((-)\otimes L^\vee)"]\\
            \cA  \arrow[r,"g"] & \DCoh(X) 
        \end{tikzcd}
    \]
    where $i\dashv r$ and $g$ is the composite equivalence $\cA \to \DCoh(X)$. Next, we define an object $P$ of $\cD$ as follows: 
    \begin{enumerate}
        \item if $g(X) \ge 1$, choose any point-like object in $\cD$; and \vspace{2mm}
        \item if $g(X) = 0$, choose the image under $i:\cA \hookrightarrow \cD$ of any point-like object. 
    \end{enumerate}
    In what follows, we sometimes omit $f$ and $g$ from the notation. In the first case, \Cref{L:pointlike} implies that up to shift $P = \pi^*\bf{C}_q$ for $q\ne p_1,\ldots, p_n$ and then $r(P) = \pi_*(\pi^*\bf{C}_q\otimes L^\vee) = \bf{C}_q$. In the second case, simply note that $r \circ i\simeq \id_{\cA}$, so that $r(P) = \bf{C}_q$ by the characterization of point-like objects of $\DCoh(\bf{P}^1)$ in \cite{BondalOrlovrecon}*{Prop. 2.2}.

    Next, let $v$ denote the class of $r(P)$ in the numerical Grothendieck group $\cN(\cA)$ of $\cA$. Choose a Bridgeland stability condition $\sigma$ on $\cA$ for which $\bf{C}_q$ is stable of phase $\phi = 1$, which can be obtained by pulling back slope stability on $\DCoh(X)$ as constructed in \cite{Br07}*{Ex. 5.6}. Then, the moduli space $\cM_\sigma(v)$ of stable objects in $\cA$ of class $v$ with respect to $\sigma$ is isomorphic to $X$. In particular, closed points of $\cM_\sigma(v)$ correspond to closed points of $X$, where the point $q\in X$ corresponds to $g^{-1}(\bf{C}_q)$. Thus, we have recovered the coarse moduli space $X$ up to isomorphism from $\cD$ and it remains to determine only which points of $\cM_\sigma(v)$ correspond to points of $X$ lying below orbifold points of $\cX$. We identify $\bf{C}_q$ with its corresponding object in $\cA$.

    Given $\bf{C}_q \in \cA$, consider the full subcategory $\cC_q\subseteq \cD$ generated under taking cones and extensions by all indecomposable objects $E$ such that $r(E) \cong \bf{C}_q$. $f(\cC_q)$ is equivalent to the full subcategory of $\DCoh(\cX)$ generated under cones and extensions by indecomposable objects $E$ such that $\pi_*(E\otimes L^\vee) \cong \bf{C}_q$, which is contained in $\cU_q$ -- indecomposability removes any summands in other cohomological degrees that are sent to zero by $\pi_*$. In particular, $f(\cC_q) \subseteq \cU_q\subseteq \Coh(\cX)$.
    
    We claim that $f(\cC_q) = \cU_q$. Since twisting by $L$ induces an autoequivalence of $\cU_q$, we may assume that $L = \cO_{\cX}$. In this case, by \Cref{C:pushforwardtorsionsheaf} the indecomposable objects $M_i$ of $\cU_q$ for $i=1,\ldots, n$ are contained in $\cC_q$ and so the simple objects $S_i$ can be obtained as cones of the natural maps $M_{i-1} \to M_i$ for $i=1,\ldots, n$. Consequently, $\cC_q = \cU_q$. Now, given $\bf{C}_{q} \in \cM_\sigma(v)$ we observe that $\rk \mathrm{K}_0(\cC_q)$ equals $e_q$. This is sufficient to recover $\cX$ up to isomorphism.
\end{proof}

\begin{rem}
    The arguments above crucially use the fact that $\cX$ is a curve throughout. Nevertheless, it would be interesting to see to what class of smooth and proper Deligne-Mumford stacks $\cY$ the Bondal-Orlov reconstruction theorem can be extended. 
    
    One might predict that for $\cY$ a root stack over a smooth projective variety $Y$ with $\omega_Y$ ample or anti-ample, the result may hold. However, if for $n\ge 2$ one takes $Y = \bf{P}^n$ and chooses a smooth degree $n+1$ hypersurface $Z$, then $Z$ is Calabi-Yau and outside of the scope of the Bondal-Orlov reconstruction theorem \cite{BondalOrlovrecon}. Consequently, taking $\cY$ to be the $n^{\rm{th}}$ root stack of $Y$ over $Z$, it seems implausible that one can recover $\cY$ from $\DCoh(\cY)$. 
\end{rem}

The following could be a worthwhile question for further investigation:

\begin{quest}
    For which classes of smooth and proper Deligne-Mumford stacks $\cY$ is it possible to prove a reconstruction theorem analogous to \Cref{T:BOstyle}?
\end{quest}

\appendix

\section{Alternative proof of nonexistence of phantoms}
\label{A:appendix}

In this appendix, we give an elementary argument for the nonexistence of phantoms for noncommutative curves from an earlier version of this paper. The merit of this argument is that the proof is logically independent of \cite{ReitenvDB} and thus more elementary.

\begin{thm}
    If $\cA$ is a noncommutative curve then $\db(\cA)$ contains no phantom subcategories. 
\end{thm}

\begin{defn}
    An Abelian category $\cA$ is called \emph{visible} if for all $X\in \cA$, $[X] = 0\in \rm{K}_0(\cA)$ implies $X = 0$.  
\end{defn}

The following simple lemma is the main observation.

\begin{lem}
\label{L:mainlemma}
     If $\cA$ is a visible Abelian category with $\rm{hd}(\cA) \in \{0,1\}$, then $\db(\cA)$ contains no phantom subcategories. 
\end{lem}

\begin{proof}
    Suppose $\cB\subset \db(\cA)$ is an admissible subcategory. Let $X\in \cB$ be given and decomposed as in \Cref{L:hd1lem}. If $X\ne 0$, then there exists an $i$ such that $H^i(X)\ne 0$ and by \Cref{L:admissiblethick} one has $\cA\cap \cB\ne 0$. If $\cB$ were a phantom then its objects would have zero image in $\rm{K}_0(\cA)$. In particular, any $X\in \cA\cap \cB$ must have trivial class in $\rm{K}_0(\cA)$. By hypothesis, this implies $\cA \cap \cB = 0$ and hence that $\cB = 0$.
\end{proof}

We next prove that examples of interest satisfy the visibility condition.  

\begin{lem}
\label{L:JordanHolder}
    If $\cA$ is a finite length Abelian category then $\cA$ is visible.
\end{lem}

\begin{proof}
    $\rm{K}_0(\cA) = \bigoplus_{S\:\text{simple}} \bf{Z}\cdot [S]$ so that for all $M\in \cA$, $[M] = \sum_{i=1}^k n_i[S_i]$ where $S_1,\ldots,S_k$ are the simples appearing in any composition series of $M$ with multiplicities $n_1,\ldots, n_k$, respectively.
\end{proof}

\begin{cor}
\label{C:hereditaryphantoms}
    If $A$ is a hereditary finite dimensional algebra over a field $k$ then $\db(\rm{mod}\:A)$ contains no phantom subcategory. 
\end{cor}

\begin{proof}
    $\rm{mod}\:A$ is finite length so we apply \Cref{L:JordanHolder} and \Cref{L:mainlemma}. 
\end{proof}

\begin{rem}
    Note that the hereditary hypothesis is necessary. Indeed, recent work of Krah \cite{krah2023phantom} which constructs phantom subcategories of $\DCoh(X)$ for $X$ the blow-up of $\bf{P}^2$ in 10 points proves that there exist (non-hereditary) finite dimensional algebras $A$ such that $\db(\rm{mod}\:A)$ contains a phantom subcategory \cite{kalck2023finite}.
\end{rem}

\begin{prop}
\label{P:projectiveDMfact}
    If $\cX$ is a projective Deligne-Mumford stack, then there exists a group homomorphism $P:\rm{K}_0(\cX)\to \bf{Q}[t]$ such that for $G\in \Coh(\cX)$, $P(G) = 0$ implies $G = 0$.
\end{prop}

\begin{proof}
    Let $\pi: \cX\to X$ denote the coarse moduli space map of $\cX$. For a locally free generating sheaf $\cE$ on $\cX$, which exists by the projectivity assumption \cite{Kresch}*{Thm. 5.3}, there is a functor $F_{\cE} :\QCoh(\cX)\to \QCoh(X)$ given by $G\mapsto \pi_*\mathcal{H}\rm{om}_{\cO_{\cX}}(\cE,G)$ \cite{Nironi09}*{Defn. 2.4}. As $\pi$ is proper, $F_{\cE}$ induces a functor $\Coh(\cX)\to \Coh(X)$. 
    
    $F_{\cE}$ is exact and for all $G\in \Coh(\cX)$, $F_{\cE}(G) = 0$ if and only if $G = 0$ by \cite{Nironi09}*{Lem. 3.4}. Fix a very ample line bundle $\cO_X(1)$ on $X$. For $G\in \Coh(X)$ let $G(m):= G\otimes_{\cO_X} \cO_X(1)^{\otimes m}$. Using this, one defines $P_{\cE}(G,m):= \chi(X,F_{\cE}(G)(m))$, which is a rational polynomial function of $m\in \bf{Z}$ and defines a polynomial $P_{\cE}(G,t)\in \bf{Q}[t]$. This gives a homomorphism $P_{\cE}(-):\rm{K}_0(\cX) \to \bf{Q}[t]$. Now, $P_{\cE}(G,t)\in \bf{Q}[t]$ is the Hilbert polynomial of $F_\cE(G)\in \Coh(X)$; so, if $P_{\cE}(G,t) = 0$, then $F_{\cE}(G) = 0$ and hence $G = 0$.
\end{proof}

\begin{cor}
\label{C:stackycurves}
    If $\cX$ is a projective orbifold curve, then $\Coh(\cX)$ is visible. Consequently, $\DCoh(\cX)$ contains no phantom subcategories. 
\end{cor}

\begin{proof}
$\Coh(\cX)$ is visible by \Cref{P:projectiveDMfact}; indeed, for $F\in \Coh(\cX)$, if $[F] = 0$, then $P_{\cE}(F) = 0$ and $F = 0$. We conclude by \Cref{L:mainlemma}.   
\end{proof}

\begin{rem}
    The condition of visibility essentially means that there are enough additive functions on $\mathrm{K}_0(\cA)$ to determine when an object is zero. When $X$ is a smooth curve which is not proper, visibility may no longer hold. For example, for $X = \bf{A}^1$ the short exact sequence 
    \[
        0\to \cO_{\bf{A}^1} \to \cO_{\bf{A}^1} \to \bf{C}_q \to 0
    \]
    implies that the class of $\bf{C}_q$ in $\rm{K}_0(\bf{A}^1)$ is zero. 
\end{rem}

In the case where $X$ is a proper curve with singularities, the category $\DCoh(X)$ becomes harder to study. In light of the results above, it seems reasonable to ask the following:

\begin{quest}
    Does there exist a proper curve $X$ such that $\DCoh(X)$ admits a phantom category?
\end{quest}

\bibliography{refs}{}
\bibliographystyle{plain}

\end{document}